\documentclass[a4paper]{amsart}
\usepackage{amsmath,amssymb,amsbsy,amsfonts,latexsym,amsopn,amstext,cite,
	amsxtra,euscript,amscd,bm}
\usepackage{amsmath,amssymb,amsbsy,amsfonts,amsthm,latexsym,
	amsopn,amstext,amsxtra,euscript,amscd,mathrsfs}
\usepackage{mathtools}
\usepackage{todonotes}
\usepackage{url}
\usepackage{float}
\usepackage{showlabels}
\usepackage[colorlinks,linkcolor=blue,anchorcolor=blue]{hyperref}
\usepackage{parskip}
\usepackage{silence}

\usepackage[english]{babel}
\usepackage{dsfont}
\usepackage{xfrac}
\usepackage {tikz}
\usepackage {geometry}
\usepackage{enumitem}
\usepackage{indentfirst}
\usepackage{graphicx}
\usepackage{caption}
\usepackage{subcaption}
\usepackage[all]{xy}
\usetikzlibrary{trees}

\makeindex \makeatletter
\def\captionof#1#2{{\def\@captype{#1}#2}}
\makeatother

\newcounter{tablegroup}
\newcommand{\tend}[3][]{\xrightarrow[#2\to#3]{#1}}

\newtheorem{thm}{Theorem}[section]
\newtheorem{cor}[thm]{Corollary}
\newtheorem{lem}[thm]{Lemma}
\newtheorem{prop}[thm]{Proposition}
\newtheorem{defn}[thm]{Definition}
\newtheorem{Que}[thm]{Question}

\numberwithin{equation}{section}
\newcommand{\C}{\mathbb C}
\newcommand{\bmu}{\bm \mu}

\newcommand{\bml}{\bm \lambda}
\begin{document}
\title[A large class of dendrite maps ]
{A large class of dendrite maps for which M\"{o}bius disjointness property of Sarnak is fulfilled.}

\author[\MakeLowercase{e.} H. \MakeLowercase{el} Abdalaoui, 
J. Devianne]{\MakeLowercase{el} Houcein \MakeLowercase{el} Abdalaoui \& Joseph Devianne}

\address{e. H. el Abdalaoui, Normandy University of Rouen,
Department of Mathematics, LMRS  UMR 6085 CNRS, Avenue de l'Universit\'e, BP.12,
76801 Saint Etienne du Rouvray - France.}
\email{elhoucein.elabdalaoui@univ-rouen.fr}
%\address{ Ghassen Askri, Department of Mathematics, College of Sciences and Humanities in Al-
%	Sulayyil, Prince Sattam Bin Abdulaziz University, Al-Sulayyil, Kingdom of
%	Saudi Arabia, and University of Carthage, Bizerte Preparatory Engineering Institute, 7021, Jarzouna, Tunisia}
%\email{askri.ghassen@gmail.com}
\address{ J. Devianne, Ecole normale sup\'erieure Paris-Saclay,
61, avenue du Président Wilson, 94230 Cachan; France;}
\email{joseph.devianne@ens-paris-saclay.fr}



\begin{abstract}
We prove that any dendrite map for which the set of endpoints is closed and countable fulfilled Sarnak M\"{o}bius disjointness. This extended a result by el Abdalaoui-Askri and Marzougui \cite{ela-GH}. We further notice that the Smital-Ruelle property can be extended to the class of dendrites with closed and countable endpoints. 
%We provide also an alternative proof of Karagulyan's result which assert that the continuous map on interval with zero entropy satisfy Sarnak M\"{o}bius disjointness.
\end{abstract}

\parskip0pt
\baselineskip11pt

\maketitle

\section{\bf Introduction}
Let $X$ be a compact metric space with a metric $d$ and let $f: X\to X$ be a continuous map.
We call for short ($X, f$) a dynamical system. The M\"{o}bius disjointness conjecture of Sarnak \cite{Sa}, \cite{Sa2} assert that for any dynamical system ($X, f$) with topological entropy zero, for any continuous function $\phi : X \longrightarrow \C$, for any point $x \in X$, we have
\begin{eqnarray}\label{first}
S_N(x,\varphi): = \frac{1}{N} \sum_{n=1}^N \bmu(n)\varphi(f^n (x)) = o(1), \textrm{as } \ N\to +\infty,
\end{eqnarray}   
where $\bmu$ is the M\"{o}bius function defined by
$$\bmu(n)  =  \begin{cases}
1  & \ \textrm{if } n = 1 \\
\bml(n) & \ \textrm{if } \textrm{all primes in decomposition of $n$ are distinct} \, \\
0 & \ \textrm{otherwise},
\end{cases}$$
Here, $\bml$ stand for the Liouville function $\bml$ given by $\bml(n) =1$ if the number of prime factors of $n$ is even and $-1$ otherwise.

We recall that the topological entropy $h(f)$ of a dynamical system $(X,f)$ is defined
as \[h(f) = \underset{\varepsilon\to 0}\lim \underset{n\to +\infty}\limsup \dfrac{1}{n}
\textrm{log}~\textrm{sep}(n, f, \varepsilon),\]
where for $n$ integer and $\varepsilon>0$, $\textrm{sep}(n, f, \varepsilon)$ is the maximal possible cardinality of
an ($n, f, \varepsilon$)-separated set in $X$, this later means that for every
two points of it, there exists $0\leq j<n$ with $d(f^j(x), f^j(y)) > \varepsilon$, $f^{j}$ stand for
the $j$-$\textrm{th}$ iterate of $f$.
 
The purpose of this paper is to strengthen the first main result in \cite{ela-GH} by proving that the M\"{o}bius disjointness of Sarnak is fulfilled for a zero topological map on the dendrites for which the set of endpoints is closed and countable. In our  proof, we extend and correct the Askri's proof of the main theorem in \cite{Askri}. We further fill the gap in the proof of Theorem 4.10 in \cite{ela-GH} when the derivative set of the endpoints does not intersect an open connected component of the canonical decomposition of dendrite.
%the due to the mistake at page 285 line 35.

Our result is also related to the conjecture of J. Li, P. Oprocha and G-H. Zhang  which assert that the spectrum of such map is discrete \cite{JOZ}.  Therefore, under this conjecture, the  M\"{o}bius disjointness property of Sarnak is fulfilled. Indeed, it is well-known that a dynamical system with discrete spectrum satisfy the  M\"{o}bius disjointness \cite{Huang}, \cite{AKLR2}. Let us further notice that recently, M. Nerurkar and the fist author proved that the M\"{o}bius disjointness of Sarnak  holds for a dynamical system with singular spectrum \cite{AM2}. At this point, let us mention that J. Li, P. Oprocha and G. Zhang  proved that any dynamical system with zero topological entropy can be embedded in a  Gehman dendrites with zero entropy.  It follows that proving the M\"{o}bius disjointness of Sarnak will settle the conjecture. We thus restrict ourselves to this setting and for the recent and other results, we  refer to \cite{ela-GH} and \cite{AM2}.
\medskip

\noindent The principal arithmetical tool of our proof is the Dirichlet prime number theorem (DPNT) \cite{Apostol}. We stress, as pointed out by e.H. el Abdalaoui and M. Nerurkar \cite{AM2}, that the only arithmetical ingredients used until now  are the following.

%\begin{enumerate}[label=(\arabic*)]
%	\item[label] description
%\end{enumerate}	
\begin{enumerate}
	\item The prime number theorem (PNT) and the Dirichlet PNT.
	\item The so-called Daboussi-Katai-Bourgain-Sarnak-Ziegler criterion. This criterion assert  that if a sequence $(a_n)$ satisfy for a large prime $p$ and $q$, the orthogonality of $(a_{np})$ and $(a_{nq})$, then the orthogonality holds between $(a_n)$ and any bounded multiplicative function. The proof of this ingredient is based on the PNT and it was used first by Bourgain-Sarnak-Ziegler \cite{BSZ} and then by many other authors. Recently, this criterion was generalized by M. Cafferata, A. Perelli and  A. Zaccagnini to some class of no bounded multiplicative functions such as the normalized Fourier coefficients $\lambda_f$ of the normalized Hecke eigenform $f$ and its Dirichlet inverse \cite{CPZ}.  
	
	\item Finally, the result of Matomaki-Radzwill-Tao on the validity of average Chowla of order two \cite{MRT}. This was used to establish the conjecture for systems with discrete spectra. It was used by el Abdalaoui-Lema\'{n}czyk-de-la-Rue  \cite{ALR}, and Huang-Wang and Zhang \cite{HWZ}.
\end{enumerate}
Let us emphasize also that the method of the authors in \cite{AM2} does not use any of the above techniques. In fact, therein, the authors presented a spectral dissection based on an unpublished result of W. Veech. From the number theory, they used Davenport's estimation combined with Daboussi's characterization of the Besicovitch class of multiplicative functions \cite{Daboussi}.

We recall that the dynamical system $(X,f)$ is said to have a the topological  discrete spectrum if the eigenfunctions of its Koopmann operator on the space $C(X)$ span a dense linear subspace of the $C(X)$ ( $C(X)$ is the space of continuous functions equipped with the strong topology). It is well-known that this is later property is equivalent to the equicontinuity of the dynamical system \cite{walter}. It also well-known that the system $(X,f)$ has always an invariant probability measure $\mu$, by  Krylov-Bogoliubov theorem \cite{walter},  that is, for any $f \in C(X)$,
$\int \phi \circ f d\mu=\int f d\mu.$ The measurable dynamical system $(X,f, \mu)$ is said to have a discrete spectrum if the eigenfunctions of its Koopmann operator on the space $L^2(X,\mu)$ span a dense linear subspace of the $L^2(X,\mu)$. In \cite{Kush}, Kusherinko established that the measurable dynamical system has a discrete spectrum  if and only if its metric sequence entropy is zero for any sequence. We recall that the system for which the sequence entropy is zero for any sequence are called null systems. It is turns out that the notion of null system is related to the so-called tame system. This later notion was coined by E. Glasner in \cite{Glasner}. The dynamical system $(X,T)$ is tame if the closure of $\big\{T^n / n \in \mathbb{Z}\big\}$ in $X^X$  is Rosenthal compact \footnote{$X^X$ is equipped with the pointwise convergence. This closure is called the enveloping semigroup of Ellis.}. We recall that the set $K$ is Rosenthal compact if and only if there is a Polish space $P$ such that $K \subset {\textrm {Baire-1}}(P)$ where  ${\textrm {Baire-1}}(P)$ is the first class of Baire functions, that is, pointwise limit of continuous functions on $P$. By Bourgain-Fremelin-Talagrand's theorem \cite{BFT}, $K$ is  Rosenthal compact if and only if $K$ is a subset of the Borel functions on $P$ with $K=\overline{\{f_n\}},$ $f_n \in C(P)$.

The precise connection between null systems and tame systems can be  stated as follows \cite{G}, \cite{Ke}, \cite{HLSY} 
$$\mathcal{N}\subset \mathcal{T},$$ where $\mathcal{N}$ is the class of null systems and $\mathcal{T}$ is the class of tame systems.

In \cite{ela-GH}, using Kusherinko's criterion, it is observed that the graph map has a discrete spectrum, for any invariant measure. In \cite{JOZ}, the authors gives an alternative proof. Therein, the authors strengthened the previous result, by establishing that the quasi-graph has a discrete spectrum, for any invariant measure. E. Glasner and M. Megrelishvili proved that every continuous action on dendron is tame \cite{Glasner}. Therein, the authors established also that every action of topological group on metrizable regular continuum is null, hence, tame. Therefore, for all those systems, the M\"{o}bius disjointness of Sarnak holds. It is also observed in \cite{ela-GH} that the M\"{o}bius disjointness property is fulfilled for the monotone maps on a local dendrites. We recall that for every topological space $X$, a map $f: ~X\to X$ is called \textit{monotone} if $f^{-1}(C)$ is connected for any connected subset $C$ of $X$. In particular, if $f$ is a homeomorphism then it is monotone.
\medskip

For the connection between the measurable notion of tame and discrete spectrum, we refer to \cite{AM}.

\medskip

Recently, G. Askri \cite{Askri2} proved that if the map on a dendrites has no Li-Yorke pairs, then its restriction to $\Lambda(f)$, the union of all $\omega$-limit sets and the set of fixed points, is uniform equicontinuous. It follows that he M\"{o}bius disjointness property hold for such maps.  
\medskip

The plan of the paper is as follows. In Section \ref{SII}, we give some definitions and preliminary properties
on dendrites which are  useful for the rest of the paper. In section \ref{SIII}, we state our main results and its consequences, we further give the proof of the main topological ingredient need it in its proof. Section \ref{SIV} is devoted to the proof of our main result. 
 Finally, in section \ref{SV}, we discuss the reduction of Sarnak M\"{o}bius disjointness to the Gehman dendrites with zero topological and its related spectral problems.  
\medskip

\section{\bf Preliminaries and some results}\label{SII}
Let $\mathbb{Z},\ \mathbb{Z}_{+}$ and $\mathbb{N}$ be the sets of integers, non-negative integers and positive integers,
respectively. For $n\in \mathbb{Z_{+}}$ denote by  $f^{n}$ the $n$-$\textrm{th}$ iterate of $f$; that is,
$f^{0}$=identity and $f^{n} = f\circ f^{n-1}$ if $n\in \mathbb{N}$. For any $x\in X$, the subset
$\textrm{Orb}_{f}(x) = \{f^{n}(x): n\in\mathbb{Z}_{+}\}$ is called the \textit{orbit} of $x$ (under $f$).
A subset $A\subset X$ is called \textit{$f-$invariant} (resp. strongly $f-$invariant)
if $f(A)\subset A$ (resp., $f(A) = A$). It is called \textit{a minimal set of $f$} if it is
non-empty, closed, $f$-invariant and minimal (in the sense of
inclusion) for these properties, this is equivalent to say that it is an orbit closure that contains no smaller one;
for example a single finite
orbit. When $X$ itself is a minimal set, then we say that $f$ is \textit{minimal}.
We define the \textit{$\omega$-limit} set of a point $x$ to be the set
$\omega_{f}(x)  = \{y\in X: \exists\ n_{i}\in \mathbb{N},
n_{i}\rightarrow\infty, \underset{i\to +\infty}\lim d(f^{n_{i}}(x), y) = 0\}$. A point $x\in X$ is called 

$-$ \textit{periodic} of period $n\in\mathbb{N}$ if
$f^{n}(x)=x$ and $f^{i}(x)\neq x$ for $1\leq i\leq n-1$; if $n = 1$, $x$ is called a \textit{fixed point} of $f$ i.e.
$f(x) = x$. 

$-$ \textit{Almost periodic} if for any neighborhood $U$ of $x$ there is $N\in\mathbb{N}$ such that
$\{f^{i+k}(x): i=0,1,\dots,N\}\cap U\neq \emptyset$, for all $k\in \mathbb{N}$. 
It is well known (see e.g. \cite{blo}, Chapter V, Proposition 5) that a point
$x$ in $X$ is almost periodic if and only if $\overline{\textrm{Orb}_{f}(x)}$ is a minimal set of $f$.

A pair $(x,y) \in X \times X$ is called
proximal if  $\liminf_{n \rightarrow +\infty}d(f^n(x),f^n(y)) = 0$;  it is called asymptotic if $\lim_{n \rightarrow +\infty }d(f^n(x),f^n(y)) = 0$.  A pair $(x,y) \in X \times X$ is
is said to be a Li-Yorke pair of $f$ if it is proximal but not asymptotic.

In this section, we recall some basic properties of graphs, tree and dendrites.

A \emph{continuum} is a compact connected metric space. An \emph{arc} is
any space homeomorphic to the compact interval $[0,1]$. A topological space
is \emph{arcwise connected} if any two of its points can be joined by an
arc. We use the terminologies from Nadler \cite{Nadler}.
\medskip

By a \textit{graph} $X$, we mean a continuum which can be written as the union of finitely many arcs
such that any two of them are either
disjoint or intersect only in one or both of their endpoints.
 For any point $v$ of $X$, the \textit{order} of $v$, 
denoted by $\textrm{ord}(v)$, is an integer $r\geq 1$ such that $v$ admits a neighborhood $U$ in $X$
homeomorphic to the set $\{z\in \mathbb{C}: z^{r}\in [0,1]\}$ with
the natural topology, with the
homeomorphism mapping $v$ to $0$. If $r\geq 3$ then $v$ is called a \textit{branch point}. If $r=1$, then we call $v$ an
\textit{endpoint} of $X$. If $r=2$, $v$ is called 
\textit{a regular point} of $X$. 

 Denote by $B(X)$ and $E(X)$ the sets of
branch points and endpoints of $X$ respectively. An edge is the
closure of some connected component of $X\setminus B(X)$, it is
homeomorphic to $[0,1]$. A subgraph of $X$ is a subset of $X$ which
is a graph itself. Every sub-continuum of a graph is a graph
(\cite{Nadler}, Corollary 9.10.1).
Denote by $S^{1}=[0,1]_{\mid 0\sim 1}$ the unit circle endowed with the
orientation: the counter clockwise sense induced via the natural
projection $[0,1]\rightarrow S^{1}$. A circle is any space homeomorphic to $S^{1}$.
\medskip
By a \textit{tree} $X$, we mean a graph which
contains no simple closed curve. A point in $X$ is a non-cut point if and only if it is an endpoint of $X$ (\cite{Nadler}, Proposition 9.27), and a continuum $X$ is a tree if and only if $X$ has only
finitely many non-cut points  (\cite{Nadler}, Theorem 9.28). As a consequence, a continuum $X$ is an arc if and only if $X$ has exactly two non cut-points (\cite{Nadler}, Corollary 9.29). 

By a \textit{dendrite} $X$, we mean a locally connected continuum
containing no homeomorphic copy to a circle. Every sub-continuum of a
dendrite is a dendrite (\cite{Nadler}, Theorem 10.10) and every
connected subset of $X$ is arcwise connected (\cite{Nadler},
Proposition 10.9). In addition, any two distinct points $x,y$ of a
dendrite $X$ can be joined by a unique arc with endpoints $x$ and
$y$, denote this arc by $[x,y]$ and let us denote by
$[x,y) = [x,y]\setminus\{y\}$ (resp. $(x,y] = [x,y]\setminus\{x\}$ and
$(x,y) = [x,y]\setminus\{x,y\}$). A point $x\in X$ is called an
\textit{endpoint} if $X\setminus\{x\}$ is connected. It is called a
\textit{branch point} if $X\setminus \{x\}$ has more than two
connected components. The number of connected components of $X\setminus \{x\}$ is called the \textit{order} of $x$ and 
denoted by ord$(x)$. The order of $x$ relatively to a subdendrite $Y$ of $X$ is denoted by $ord_Y(x)$.
Denote by $E(X)$ and $B(X)$ the sets of
endpoints, and branch points of $X$, respectively.
By (\cite{Kur}, Theorem 6, 304 and Theorem 7, 302), $B(X)$ is at most countable. A point $x\in
X\setminus E(X)$ is called a \textit{cut point}. It is known that the  set of cut
points of $X$ is dense in $X$ (\cite{Kur}, VI, Theorem 8, p. 302).
Following (\cite{Ar}, Corollary 3.6), for any dendrite $X$, we have
B($X)$ is discrete whenever E($X)$ is closed.  An arc $I$ of $X$ is called \emph{free} if $I \cap B(X)=\emptyset$.
For a subset $A$ of $X$, we call \emph{the convex hull} of $A$, denoted by $[A]$, the intersection of all
sub-continua of $X$ containing $A$, one can write $[A] = \cup_{x, y\in A}[x,y]$.   

\medskip
A continuous map from a  dendrite (resp. quasi-graph, graph )  into itself is called a\textit{ dendrite map} (resp. \textit{quasi-graph map}, resp. \textit{graph map}). 

\medskip
It is well known that every dendrite map has a fixed point (see \cite{Nadler}). If $Y$ is a sub-dendrite of $X$, define the retraction (or the \textit{first point map}) $r_{Y} : X \rightarrow Y$
by letting $r_{Y}(x) = x$, if $x\in Y$, and by letting $r_{Y}(x)$ to be the unique point $r_{Y}(x)\in Y$ such that
$r_{Y}(x)$ is a point of any arc in $X$ from $x$ to any point of $Y$ (see \cite[Lemma 10.24,p. 176]{Nadler} %\label{first_point_map}, p. 176). 
Note that the map $r_{Y}$ is constant on each connected component of $X\backslash Y$.

Furthermore, any dendrite can be approximated by tree, that is, there exist an increasing sequence of tree $(Y_i)$ such that $\lim_{i}Y_i=X$ and the sequence of first point maps $(r_{Y_i})$ converges uniformly to the identity map on $X$
(see \cite{Nadler}, Theorem 10.27)
\medskip
As customary, for a subset $A$ of $X$, denote by $\overline{A}$ the closure of $A$, $A'$ the derivative set, that is, the set of all accumulation points of $A$ and by $\textrm{diam}(A)$ the diameter of $A$. For the proprieties of derivative set, we refer to \cite[$\S$ 9, pp.75-78 ]{KurI}. 

%By a \textit{local dendrite} $X$, we mean a continuum every point of which has a dendrite neighborhood. 
%A local dendrite is then a locally connected continuum containing only a finite number of
%circles (\cite{Kur}, Theorem 4, p. 303). As a consequence every sub-continuum of a local
%dendrite is a local dendrite (\cite{Kur}, Theorems 1 and 4, p. 303). Every graph and every dendrite is a local
%dendrite.
%A continuous map from a local dendrite (resp. graph, resp. dendrite) 
%into itself is called a\textit{ local dendrite map} (resp. \textit{graph map}, resp. \textit{dendrite map}).
%It is well known that every dendrite map has a fixed point
%(see \cite{Nadler}). If $A$ is a sub-dendrite of $X$, define the retraction (or the \textit{first point map}) $r_{A} : X \rightarrow A$
%by letting $r_{A}(x) = x$, if $x\in A$, and by letting $r_{A}(x)$ to be the unique point $r_{A}(x)\in A$ such that
%$r_{A}(x)$ is a point of any arc in $X$
%from $x$ to any point of $A$, if $x\notin A$ (see \cite{Nadler}, p. 176).
%Note that the map $r_{A}$ is constant on each
%connected component of $X\backslash A$.
%For a subset $A$ of $X$, denote by $\overline{A}$ the closure of $A$ and by $\textrm{diam}(A)$ the diameter of
%$A$. \\
%For every topological space $X$, a map $f: ~X\to X$ is called \textit{monotone} if $f^{-1}(C)$ is connected for any connected
%subset $C$ of $X$. In particular, if $f$ is a homeomorphism then it is monotone. Notice that when $X$ is a dendrite, the map $r_{A}$ (above) is monotone. Finally, as customary, we will denote by $|X|$ the cardinal of $X$.
%\medskip
\medskip
\noindent We need the following results on the structure of dendrites.

\begin{prop}[\cite{Nadler}]
	Every sub-continuum of a dendrite is a dendrite.
	Every connected subset of $X$ is arcwise connected.
\end{prop}

%\begin{prop}[\cite{Nadler}, Lemma 10.24]\label{first_point_map}
%	Let $X$ be a dendrite, and let $Y$ be a subcontinuum of $X$. For each $x \in X \backslash Y$, there is a unique point $r(x) \in Y$ such that $r(x)$ is a point of any arc in $X$ from $x$ to any point of $Y$. This map is continuous and is called the first point map of $Y$.
%\end{prop}
%Notice that 

\begin{prop}[\cite{Nadler}, Theorem 10.2]
	The set of all branch points is countable.
\end{prop}

\begin{prop}[\cite{Nadler}, p.187, 10.41 ]\label{Cut-dense}
	Let $X$ be a dendrite. The set of points of order $2$ is continuumwise dense in $X$.
\end{prop}

\begin{thm}[\cite{Nadler}, p. 188, 10.42]\label{decomposition_nadler}
	Let X be a dendrite. Then for each $\delta>0$, there are finitely many open connected open and pairwise disjoint subsets $U_1, \dots, U_r$ such that
	\begin{enumerate}
		\item $ X = \bigcup_{i=1}^r \overline{U_i}$
		\item $\mathrm{diam}(U_i) < \delta$
		\item $card (U_i \cap U_j) \leq 1$
	\end{enumerate}
\end{thm}
For the proof of Theorem \ref{decomposition_nadler} see also the proof of \cite[Theorem 5]{Kur}.
\begin{defn}
	Let $X$ be a dendrite and $F \subset X$. We set $[F]$ the convex hull of $F$ given by the intersection of all subdendrites containing $F$.
\end{defn}
The following lemma gives a better description of convex hulls in dendrites.
\begin{lem}[\cite{Askri}]\label{dendrites_convexhull}
	Let $X$ be a dendrite and $F$ a non empty closed subset of X. Let $a \in F$, then
	\begin{itemize}
		\item $[F] = \bigcup _{z \in F} [a,z]$
		\item $E([F])\subset F$ and we have $E([F]) = F$ when $F\subset E(X)$
	\end{itemize}
	where $[A]$ stand for the convex hull of a subset $A$.
\end{lem}
We need also the following lemma stated with $E(X)'$ finite in \cite{Askri}. Its proof can be obtained in the similar manner as in \cite{Askri}. %(see also \cite{Askri2} where the lemma is localized for a minimal subset.). 
\begin{lem}[\cite{Askri}]\label{connected_components}
	Let $X$ be a dendrite such that $E(X)$ is closed. Let $Y \subset X$ a subdendrite of $X$ such that $E(Y) \cap E(X)' =\emptyset $. Then $Y$ is a tree and $X \backslash Y$ has finitely many connected components. Furthermore, there are pairwise disjoint subdendrites $D_1, \dots, D_n$ in $X$ such that $X \backslash Y  \subset \bigcup_{1 \leq i \leq n} D_i$ and $D_i \cap Y$ is reduced to one point for each $i$.
\end{lem}

\section{Main results}\label{SIII}

The main result of this paper is the following.
\begin{thm}\label{main}Let $f$ be a dendrite map on $X$ with zero topological entropy for which $E(X)$ is closed and countable. Then, $f$ satisfy Sarnak M\"{o}bius disjointness \eqref{first}. 
\end{thm}

As a consequence, we have 

\begin{cor}[\cite{ela-GH}, Theorem 4.10.]Let $f$ be a dendrite map on $X$ with zero topological entropy for which $E(X)$ is closed and $E(X)'$ is finite. Then, $f$ satisfy Sarnak M\"{o}bius disjointness \eqref{first}. 
\end{cor}

\begin{proof}By the characterization theorem of the dendrites with a countable set of end points \cite[Theorem 5]{Cheratonik}, if $E(X)$ is not countable then it is a Cantor set. Therefore $E(X)'$ is not finite. This finished the proof of the corollary.
\end{proof}
\medskip
\medskip

\noindent At this point, let us notice that recently, G. Askri \cite{Askri2} proved the following lemma. 

\begin{lem} Let $X$ be a dendrite with $End(X)$ countable and closed and
let $f : X \longrightarrow X$ be a continuous map. Suppose that $f$ has no Li-Yorke pairs.
If one of the following assumptions holds
\begin{enumerate}
\item the collection of minimal sets is closed in $(2^X , d_H )$,
\item $f|_{\Lambda(f)}$ is equicontinuous at every point in $End(X)'$ ,
\item $f |_{P (f )}$ is uniform equicontinuous,
\end{enumerate}
then $f|_{\Lambda(f)}$ is equicontinuous.
where $\Lambda(f)$ stand for the union of all $\omega$-limit sets a the set of fixed points, and $P(f)$ is the set of periodic points.
\end{lem}

\noindent It follows that if the dendrite map has no Li-Yorke pairs and one of the previous assumptions holds the M\"{o}bius disjointness property is fulfilled. This gives an alternative proof to the proof of D. Karagulan \cite{Ka}.
We further notice that there exist a dendrite map without Li-York pair such that   $\Lambda(f)=P(f)$ and he collection of minimal sets is not closed in $(2^X , d_H)$. Therefore, the restriction of $f$ to $\Lambda(f)$ is not equicontinuous (see \cite[Lemma 4]{Askri2}).
\medskip
\medskip
%\end{rem}  
\
noindent{}For the proof of our main result (Theorem \ref{main}), we need to extend Smital-Ruette property. This later property is implicitly contained in several papers of Sharkovsky, but its first proof can be found in \cite{smital} and \cite{ruette}. G. Askri extended this property to the class of dendrites for which the set of endpoints is closed and its derivative set is finite.  Here, we extend it to the dendrites with closed and countable endpoints.  Our extension is largely inspired from that of G. Askri.  We state it as follows.
%\begin{prop}\label{ergo-coutable}
%	Let $(X,\A,f,\mu)$ be a measurable dynamical system. Assume that its set of the ergodic component is countable and for each ergodic component the  spectrum is discrete. Then,  $(X,\A,f,\mu)$ has a discrete spectrum.
%\end{prop}
%\begin{proof} Let $F \in L^2(X,\mu)$. Then, by the ergodic decomposition,
%	$$\sum_{e \in D}a_e \int |F(x)|^2 d\mu_e< +\infty.$$
%	Whence $F$ is in $L^2(X,\mu_e)$, for each ergodic component $\mu_e$. Therefore 
%	\linebreak $\overline{\Big\{T^nF/n \in \Z\Big\}}$ is compact in   $L^2(X,\mu_e)$. We thus get that for each sequence in $\overline{\Big\{T^nF/n \in \Z\Big\}}$ there is a subsequence which converge in $L^2(X,\mu_e)$. Applying the Cantor diagonal method, we obtain a subsequence which converge in
%$L^2(X,\mu)$. We can thus conclude, by appealing to Proposition \ref{p310}, that $(X,\A,f,\mu)$  has a discrete spectrum. The proof of the lemma is complete.	
%\end{proof}
\begin{thm}\label{Askri}
	Let X be a dendrite such that $E(X)$ is closed and countable. Let $f: X \to X$ be a dendrite map with zero topological entropy. Let $L:= \omega_f(x)$ be an uncountable $\omega$-limit set. Then there is a sequence of $f$-periodic subdendrites $(D_k)_{k \geq 1}$ and a sequence $(n_k)_{k \geq 1}$ of integers with the following properties for any $k$,
	\begin{enumerate}
		\item $D_k$ has period $\alpha_k := n_1 n_2 \dots n_k$.
		\item For $i \neq j \in \{0,\dots,,\alpha_k-1\}, f^i(D_k)\cap f^j(D_k)=\emptyset$.
		\item $ L \subset \bigcup _{i=0}^{\alpha_k -1}f^i(D_k)$.
		\item $\bigcup_{k=0}^{n_j-1}f^{k\alpha_{j-1}}(D_j)\subset D_{j-1}$.
		\item For $ i \in \{0,\dots,\alpha_k-1\}, f(L\cap f^i(D_k))=L\cap f^{i+1}(D_k)$.
	\end{enumerate} 
\end{thm}
\noindent For sake of completeness, we present a sketch of the proof of Theorem \ref{Askri}. For that, we proceed by induction and as in the proof of Smital-Ruelle property, we are going to prove the result at step $k=1$ and then extend it by induction. 

\begin{prop}\label{askri_step1}
	With the same assumptions as in Theorem \ref{Askri}, there exists a connected subset $J$ of $X$ and an interger $n \geq 2$ such that
	\begin{enumerate}[label = (\roman*)]
		\item $J, f(J), \dots f^{n-1}(J)$ are pairwise disjoint.
		\item $f^n(J)=J$.
		\item $ L \subset \bigcup_{0\leq i \leq n-1} f^i(J)$.
		\item For $ i \in \{0,\dots,n-1\}, f(L\cap f^i(J))=L\cap f^{i+1}(J)$.
	\end{enumerate}
\end{prop}

\noindent The proof of Proposition \ref{askri_step1} relies on the following Lemma.
%which is an analog of Porposition \ref{periodic} and which is also proved in \cite{Askri}. 

\begin{lem}\label{dendrite_periodic}
	For $x\in X$, if $\omega(f)(x)=L$ is infinite then it contains no periodic point of $f$.
\end{lem}
\noindent{}As in the proof on the interval, we are going to build horseshoes to get contradictions. the following lemma will be useful.
\begin{lem}\label{horsehoe_dendrites}
	Let $f: X \to X$ a dendrite map such that $E(X)$ is closed and countable. Let $a$ a fixed point for $f$ and $L:=\omega_f(x)$ an uncountable $\omega$-limit set such that $L \cap P(f)= \emptyset$ then for any $y \in L$, there is $p,k \geq 0$ such that $[a,f^k(x)]\subset[a,f^p(y)]$.
\end{lem}
\begin{proof}
	Let $ y \in L$. We show first that $\omega_f(y)\subset\omega_f(x)$ is uncountable.
	Since it is closed and invariant, we can consider a subset $K \subset \omega_f(y)$ which is minimal. Since $L \cap P(f) = \emptyset$, it contains no periodic point, so it is infinite and with no isolated point, hence it is uncountable and so it is for $\omega_f(y)$.\newline
	Denote now by $(C_i)_{i \in \mathbb{N}}$ the sequence of connected components of $X \backslash (B(X) \cup E(X)))$. By Lemma \ref{connected_components}, each $C_i$ is an open free arc in $X$. There is $j \in \mathbb{N}$ such that $\omega_f(y) \cap C_j$ is uncountable. Let $v,v$ in this intersection such that $u \in (a,v)$. Let $I_u, I_v$ two open dsijoint arcs in $C_j$ such that $u \in I_u, v\in I_v$. Then there exist $p,k' \in \mathbb{N}$ such that $f^p(y) \in I_v$ and $f^{k'}(y) \in I_u$. but since $f^{k'}(y) \in L$, there is $k \in \mathbb{N}$ such that $f^k(x) \in I_u$ which leads to the inclusion $[a, f^k(x)]\subset [a,f^p(y)]$.
\end{proof}
\medskip
\medskip

\begin{proof}[\textbf{Proof of Lemma \ref{dendrite_periodic}.}]
	Let $M= [L]$ the convex hull of $L$. We only focus on the case $M \cap Fix(f) \neq \emptyset $. Let $a \in M \cap Fix(f)$. By the Theorem \ref{dendrite_periodic}, we have that $a \in M \backslash L \subset M \backslash E(M)$ by Lemma \ref{connected_components}. Hence, $ord(a,M) \geq 2$.
\end{proof}
	
\noindent	The proof of the first step $k=1$ for the Smital-Ruelle property is basic since the jumping dynamic from one side of the fixed point to the other can be described. But, here since each branch connected to the fixed point $a$ is also a dendrite complicates the proof. Indeed the natural idea would be to set $\omega_i = \omega_f(x) \cap C_i$ with $\{C_i\}_{i=1,\dots,p}$ the connected components of $X \backslash \{a\}$ and then to apply  the same reasoning as in the proof of the first step of the Smital-Ruelle property. It would be then very complicated to prove that $a \notin J$. We need then to take apart the preimages of $a$, construct sets in the same way we constructed $\omega_0$ and $\omega_1$, but such that they are disjoint from these preimages and then construct the set $J$.
\medskip
\medskip	
	Let us set $F_a = \bigcup_{n \geq 0} f^{-n}(a)$ and $Y_a= [F_a]$. We thus have the following lemma.
	\begin{lem}
		$L \subset M \backslash Y_a$ and $M\backslash Y_a$ has finitely many connected components.
	\end{lem}
\begin{proof}
	We start first by proving that $[a,z] \cap L =\emptyset$ for $z \in F_a$. Indeed, if $F_a= \{a\}$, there is nothing to prove. Else, suppose contrary that there is $y \in [a,z]\cap M, f^p(z)=a, p \geq 1$. We set $I= [a,y]$ and $J=[y,z]$ and easily construct an horseshoe with the help of Lemma \ref{horsehoe_dendrites}. It is actually possible to prove something stronger, that is, $L\cap \overline{F_a}=\emptyset$ (see \cite{Askri}).
\medskip
\medskip
This proves that $ L \subset M \backslash   \bigcup _{z \in \overline{F_a}} [a,z]$. Hence , by Lemma \ref{dendrites_convexhull}, we have $ L \subset M \backslash Y_a$. Again by Lemma \ref{dendrites_convexhull}, $E(Y_a) \subset \overline{F_a}$ and $E(M)'\subset E(M) \subset L$ then $E(Y_a) \cap E(M)' =\emptyset$. Hence, with Lemma \ref{connected_components}, $Y_a$ is a tree and $M\backslash Y_a$ has finitely many connected components.
\end{proof}	
	
\noindent We are now able to have a better description of $M$ and understand the action of $f$ on it. In fact we have the following lemma.
\begin{lem}	Under above notations, we have\begin{enumerate}
		\item  \label{One}If $F_a \neq \{a\}$, for any $y \in L, (a,y] \cap F_a \neq \emptyset$.
		\item  \label{Two} $M \backslash F_a$ has finitely many connected components $C_1, \dots, C_n$ intersecting $L$. We denote by $l_k= L \cap C_k$ each intersection.
		\item \label{Three} For each  $k \in \{1,\dots,n\}$, there is a unique $j:=\sigma(k) \in \{1,\dots,n\}$ such that $f(C_k)\cap C_j \neq \emptyset$ and we have $f(l_k) = l_{\sigma(k)}.$
		\item  \label{Four} Each $l_k$ is clopen relatively to $L$.
		\item  \label{Five} $\sigma$ is an $n$-cycle.
	\end{enumerate}
\end{lem}
\begin{proof}We give only  the proof of \eqref{Three}. Since, the proof of \eqref{One}, \eqref{Two}, \eqref{Four} and \eqref{Five} is the same as the proof of lemma 5.10 in \cite{Askri}. Let $k \in \{1,\dots,n\}$, then  $f(l_k) \subset f(L)=L=\bigcup_{1\leq i\leq n} l_i.$ Therefore, there is $j$ such that $f(l_k)\cap l_j\neq \emptyset$ then $f(C_k)\cap C_j \neq \emptyset$. Suppose now that there is $i \neq j$ such that $f(C_k)\cap C_i\neq \emptyset$. Let $x_j \in f(C_k)\cap C_j$ and $x_i \in f(C_k)\cap C_i$. We have $[x_j,x_i] \subset f(C_k)$ and $f(C_k) \cap F_a =\emptyset$ so $[x_j,x_i]\cap F_a =\emptyset$. Hence $H= C_j \cup[x_j,x_i]\cup C_i$ is a connected subset of $M$ such that $H \cap F_a =\emptyset$. By maximality of $C_i$ and $C_j$ we have $H=C_i=C_j$ which gives a contradiction.
\medskip

	Let $j = \sigma(k)$, then we have $$f(l_k) \subset l_{\sigma(k)}.$$ Suppose that $f(l_k) \subsetneq l_{\sigma(k)}$. Then there exists $i\neq j$ such that $f(l_i) \cap l_{\sigma(k)}$ since $f(L)=L$.  Consequently, $f(l_i) \subset l_{\sigma(k)}$ and $\sigma(i)=\sigma(k)$. Whence, $\sigma$ is not injective and not surjective and we have $f(L) \subset \bigcup_{r=1}^n l_{\sigma(r)} \subsetneq L$ which gives a contradiction and proves that $f(l_k) = l_{\sigma(k)}$. 
\end{proof}
\medskip
\medskip	
	
	%Without a loss of genrality, we can thus suppose that $\sigma$ is the natural $n$ cycle $ ( 1 \: 2 \: \dots \: n)$ ie $\sigma^i(1)=i+1$ for $ i \leq n-1$ and $\sigma^n(1)=1$.\newline
\noindent We are now able to give the proof of  Proposition \ref{askri_step1} and Theorem \ref{Askri}.
\medskip

\begin{proof}[\textbf{Proof of Proposition \ref{askri_step1} and Theorem \ref{Askri}.}]
Let $s_1 := [l_1]\subset C_1$ and  $J:= \bigcup_{k=0}^{\infty}f^{kn}(s_1)$. We proceed to establish that $J$ satisfy the conditions of the Proposition \ref{askri_step1}.
\medskip
\medskip
\noindent We notice first that  $J$ is connected as a limit of an increasing sequence of connected sets. Since  $s_1$ is a subdendrite and $s_1 \subset f^{n}(s_1)$. 
\medskip
\noindent For the proof of $(i)$, we use here the construction of $F_a$ and $C_i$ as maximal connected components to show that $J, f(J),\dots,f^{n-1}(J)$ are pairwise disjoint. Indeed, since $s_1 \cap F_a$ and $F_a$ is backward invariant,  $f^i(s_1) \cap F_a = \emptyset$ for $ i \geq 0$ and consequently, $f^k(J) \cap F_a = \emptyset $ for $k \geq 0$.
\medskip
\noindent Suppose now that $f^i(J) \cap f^j(J) \neq \emptyset$, then $f^i(J) \cap f^j(J)$ is  connected. If we take then $u \in l_{\sigma^i(1)}$ and $v \in l_{\sigma^j(1)}$ then $K =C_{\sigma^i(1)} \cup [u,v] \cup C_{\sigma^j(1)}$ is connected in $M$ and disjoint with $F_a$ by maximality, $K=C_{\sigma^i(1)}=C_{\sigma^j(1)}$ and then, $i=j$.
\medskip
	$(ii)$ and $(iii)$ are immediate.
\medskip
\noindent Finally, for	$(iv)$, we  have $f(L \cap f^i(J)) = f(l_{\sigma^i(1)}) = l_{\sigma^{i+1}(1)} = L \cap f^{i+1}(J)$.
\medskip
\noindent We conclude the proof of the theorem by induction.
\end{proof}
\medskip
\medskip
\medskip

\noindent{}Let us proceed now to the proof of our first main result.

\section{Proof of The main result.}\label{SIV}

As we mention in introduction, our proof is based essentially on DPNT which we formulated as follows.
 
%\noindent\emph{\textbf{Proof of Theorem \ref{main}.}}. Let $\mu$ be an invariant measure on $X$ and consider its ergodic decomposition $\{\mu_e\}_{e \in E}$. Let 
%$\textrm{supp}(\mu_e)$ be the support of the ergodic measure $\mu_e$. By lemma \ref{Veech}, there is a point $x_e$ in a Borel set with full measure with respect to $\mu_e$ and dense orbit in $\textrm{supp}(\mu_e)$. But the ergodic measures are either equivalent or singular. Then, we can assume that the family $\{\mu_e\}_{e \in E}$ contain only a mutually singular measures. Hence, the points $(x_e)_{e \in E}$ are all distinct. Now, by taking $Y_e=[\textrm{supp}(\mu_e)]$ and by appealing to Lemma \ref{Askri}, we can assume that 
%$(E(Y_e))_{e \in E}$ are mutually disjoint. Moreover, by Lemma \ref{l14}, we have that for each $e \in E$, the set $E(Y_e)$ is countable. It is also easy to see that for each $e \in E$, we have
%$Y_e \subset Y \setdef [\textrm{supp}(\mu)].$ Whence, by Lemma \ref{l12},   
%$\displaystyle \bigcup_{e \in E}E(Y_e)' \subset E(X)'$ is countable. But, if the set $E$ is not countable then $\displaystyle \bigcup_{e \in E}E(Y_e)'$ is not countable which is impossible by our assumption. We thus deduce that $E$ is a countable and we conclude by Proposition \ref{ergo-coutable} that 
%the dynamical system $(X,T,\mu)$ has a discrete spectrum.
%\endproof
\medskip
\begin{prop}\label{periodic}
	If $(x_n)_{n \in \mathbb{N}}$ is eventually periodic sequence, 
	$$ \frac{1}{N} \sum_{n=1}^N \mu (n) x_n \underset{N\to \infty}{\longrightarrow} 0. $$
\end{prop}
As a consequence, we have that if the $\omega$-limit set $\omega_{f}(x)$ is finite then for any continuous $\varphi$, we have that 
$$S_N(x,\varphi) \tend{N}{+\infty}0.$$
Indeed, it is easy to see that there exist a periodic point $y$ such that $(x,y)$ is asymptotic pair, that is,
$\lim_{n \rightarrow +\infty }d(f^n(x),f^n(y))=0.$  
We thus need to verify \eqref{first} only when the $\omega_{f}(x)$ is infinite. For that, 
We need also the following lemma from \cite{Ka}.
\begin{lem}\label{holed}
	Let $(a_n) \in [0,1]^\mathbb{N}$ such that there exists $n_0 \in \mathbb{N}$ such that for $n,m \in \mathbb{N} $ if $a_n \neq 0$ and $a_m \neq 0$ then $n=m$ or $|n-m|\geq k$. Then
	$$\underset{N \to \infty}{\limsup} \frac{1}{N}\left|\sum_{n=1}^N a_n \right| \leq \frac{1}{k}$$
\end{lem}
\noindent At this point, let us assume that $L=\omega_{f}(x)$ is infinite. We need thus to examine two cases according to $L$ is countable or not. In the first case, it follows that $Y=\overline{\textrm{Orb}}(x)={\textrm{Orb}}(x) \cup L$ is countable, Hence, $(Y,f|_{Y})$ is tame and  \eqref{first} holds by \cite[Theorem 1.8]{HWZ}. We can thus suppose $L$ uncountable.
\medskip

\noindent As in the proof given by D. Karagulyan \cite{Ka} and e. H. el Abdalaoui, G. Askri and H. Marzougui \cite{ela-GH}, we are going to make a proof by an approximation by step functions. Let $\varepsilon >0$ and let $\delta >0$ an uniform continuity modulus adapted to $\varphi$. We are going to construct a decomposition of $X$ into sets of diameter $< \delta$ and such that each set shares at most two points with the rest of $X$.

\noindent We use Proposition \ref{decomposition_nadler} to get a first decomposition. There exists open connected subsets $V_i, i=1,\dots, r$ such that
\begin{enumerate}
	\item $ X = \bigcup_{i=1}^r \overline{V_i}$
	\item \label{ii}$\mathrm{diam}(V_i) < \delta$
	\item $card (V_i \cap V_j) \leq 1$ if $i\neq j$
\end{enumerate}
%To construct it, we consider $F$ a finite set of cutpoints such that each component of $X \backslash F$ has a diameter  $< \delta$ and we consider $U_i$ the connectd components of $X \backslash F$.\newline
\noindent To perform such decomposition, we apply Proposition \ref{Cut-dense} to see that there is a finite set $F$ of cut points in $X$ such that each component $V_i$ of $X \setminus F$ \eqref{ii} is satisfied (see \cite[p.188]{Nadler}) and (see also the proof of \cite[Theorem 5. p 302]{KurI}).

\noindent Now, in order to get the same kind of estimate  as in the proof by D. Karagulyan \cite{Ka} and e. H. el Abdalaoui, G. Askri and H. Marzougui \cite{ela-GH}, we need to modify the sets $V_i$ such that each $V_i$ share at most two points with the rest of the decomposition.

%If we want to apply the sentence  "Otherwise, there is at most two distinct numbers $0 \leq s, r \leq \alpha_k$ − 1 such that $f^s (D_k ) \not\subset U_j , f^ s(D_k ) \cap U_j  \neq \emptyset, f^r (D k ) \not\subset U_j$ and$ f^r (D_k ) \cap U_j \neq \emptyset$, we have to transform the sets $V_i$ such that 

%$$ \mathrm{card}(\overline{V_i} \backslash V_i) \leq 2$$

\noindent We set $i \in\{1,\dots, r\}$ and consider the set $V_i$.
\medskip
\noindent We put $\partial V_i =\{f_1,\dots,f_{p_i}\}$ with $f_1,\dots,f_{p_i} \in F$. We consider now the tree which has endpoints the $f_1,\dots,f_{p_i}$.
$$ T = \bigcup_{1\leq i,j \leq p} [f_i,f_j] \subset \overline{V_i}.$$
We split $T$ into arcs by writing
$$ T= \bigcup_{j=1}^m \overline{a_j},$$
with $a_j$ disjoint open arcs, $a_j \cap F=\emptyset$ and  $ card(\overline{a_j} \cap F) = card(\overline {a_j} \cap \{f_1,\dots,f_p\})\leq 2.$
%and for $i,j \in \{1,\dots,m\}, card(\overline{a_i} \cap \overline{a_j}) \leq 1$

\noindent We set now $(C_k)_{ k \in K}$ the open connected components of $V_i \backslash T$ and let $r : \overline{V_i} \to T$ the first point map of $T$ in $V_i$. Since $r$ is locally constant, $r$ is constant on each $C_k$ and we denote by $x_k \in T$ the value on $C_k$. Since $\{f_1,\dots, f_p\}$ are cutpoints, then $x_k \in T \backslash \{f_1,\dots, f_p\}$. Hence, there exists a unique $j= c(k) \in \{1,\dots,m\}$ such that $x_k \in a_{j_k}$. We have then a function  $ c : K \to \{1,\dots,m\}$ such that for any $k \in K$, and $y \in C_k$, $r(y) \in a_{c(k)}$.

\noindent We finally set for $j \in \{1,\dots,m\}$,
$$ b_j= a_j \cup \bigcup_{c(k)=j} C_k. $$

\noindent Then every $b_j$ is connected since every point in $b_j$ is connected to $x_{c(j)}$, and they are all pairwise disjoint by construction. We further have that $\overline{b_i} $ and $\overline{b_j}$ are either disjoint or intersect on one of their endpoints. Furthermore, since $\partial C_k = x_{c(k)} \in a_{c(k)}$, we have
$$ \partial b_j = \partial a_j $$

\noindent Whence, we have the desired  property , that is,
$$card(\partial b_j) = card(\partial a_j) \leq 2,$$

and at the same time, we have
$$\overline{V_i} = \bigcup_{j=1}^m \overline{b_j}$$
\medskip
\noindent We explain the construction with these drawings. 
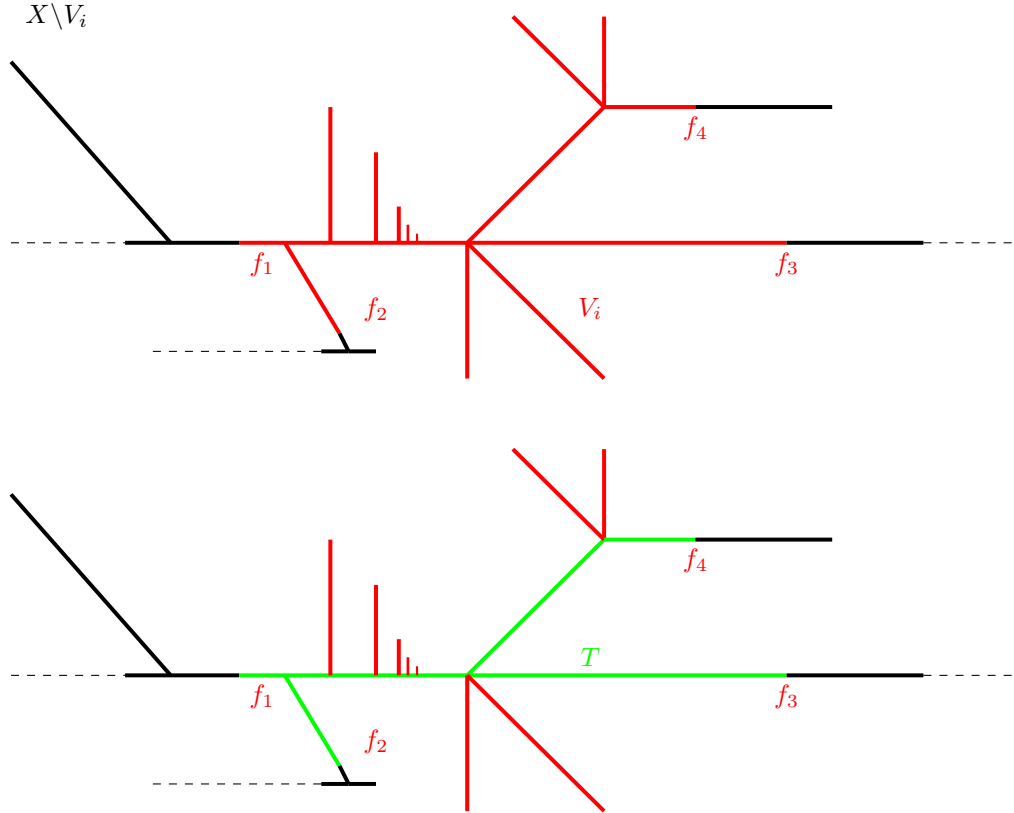
\begin{figure}[ht!]
	\begin{center}
		\begin{tikzpicture}[scale=6]

		% Axes et point d

		\draw[dashed] (0,0) -- (0.25,0);
		
		\draw[line width = 1.5pt] (0.25,0) -- (0.5,0);
		\draw[line width = 1.5pt] (0,0.4) -- (0.35,0);
		\draw[line width = 1.5pt,color=red] (0.5,0) -- (1.7,0);
		\draw[line width = 1.5pt] (1.7,0) -- (2,0);
		\draw[dashed] (2,0) -- (2.2,0);

		\draw[line width = 1.5pt,color=red] (0.7,0) -- (0.7,0.3);
		\draw[line width = 1.5pt,color=red] (0.8,0) -- (0.8,0.2);
		\draw[line width = 1.5pt,color=red] (0.85,0) -- (0.85,0.08);
		\draw[line width = 1pt,color=red] (0.87,0) -- (0.87,0.04);
		\draw[line width = 0.8pt,color=red] (0.89,0) -- (0.89,0.02);
		
		\draw[line width = 1.5pt,color=red] (0.6,0) -- (0.72,-0.2);
		\draw[line width = 1.5pt] (0.72,-0.2) -- (0.74,-0.24);
		\draw[line width = 1.5pt] (0.74,-0.24) -- (0.68,-0.24);
		\draw[line width = 1.5pt] (0.74,-0.24) -- (0.8,-0.24);
		\draw[dashed] (0.68,-0.24) -- (0.3,-0.24);

		\draw[line width = 1.5pt,color=red] (1,0) -- (1.3,0.3);
		\draw[line width = 1.5pt,color=red] (1,0) -- (1.3,-0.3);
		\draw[line width = 1.5pt,color=red] (1,0) -- (1,-0.3);
		\draw[line width = 1.5pt,color=red] (1.3,0.3) -- (1.3,0.5);
		\draw[line width = 1.5pt,color=red] (1.3,0.3) -- (1.1,0.5);
		\draw[line width = 1.5pt,color=red] (1.3,0.3) -- (1.5,0.3);
		\draw[line width = 1.5pt] (1.5,0.3) -- (1.8,0.3);

		\draw(0.55,0) node[below,color = red]{$f_1$};
		\draw(0.8,-0.1) node[below,color = red]{$f_2$};
		\draw(1.7,0) node[below,color = red]{$f_3$};
		\draw(1.5,0.3) node[below,color = red]{$f_4$};

		\draw(1.27,-0.1) node[below,color=red]{$V_i$};
		
		\draw(0.10,0.45) node[above] {$ X \backslash V_i$};

		%Projets
				\end{tikzpicture}
	\end{center}
	
\end{figure}

\begin{figure}[ht!]
	\begin{center}
		\begin{tikzpicture}[scale=6]

		% Axes et point d

		\draw[dashed] (0,0) -- (0.25,0);
		
		\draw[line width = 1.5pt] (0.25,0) -- (0.5,0);
		\draw[line width = 1.5pt] (0,0.4) -- (0.35,0);
		\draw[line width = 1.5pt,color=green] (0.5,0) -- (1.7,0);
		\draw[line width = 1.5pt] (1.7,0) -- (2,0);
		\draw[dashed] (2,0) -- (2.2,0);

		\draw[line width = 1.5pt,color=red] (0.7,0) -- (0.7,0.3);
		\draw[line width = 1.5pt,color=red] (0.8,0) -- (0.8,0.2);
		\draw[line width = 1.5pt,color=red] (0.85,0) -- (0.85,0.08);
		\draw[line width = 1pt,color=red] (0.87,0) -- (0.87,0.04);
		\draw[line width = 0.8pt,color=red] (0.89,0) -- (0.89,0.02);
		
		\draw[line width = 1.5pt,color=green] (0.6,0) -- (0.72,-0.2);
		\draw[line width = 1.5pt] (0.72,-0.2) -- (0.74,-0.24);
		\draw[line width = 1.5pt] (0.74,-0.24) -- (0.68,-0.24);
		\draw[line width = 1.5pt] (0.74,-0.24) -- (0.8,-0.24);
		\draw[dashed] (0.68,-0.24) -- (0.3,-0.24);

		\draw[line width = 1.5pt,color=green] (1,0) -- (1.3,0.3);
		\draw[line width = 1.5pt,color=red] (1,0) -- (1.3,-0.3);
		\draw[line width = 1.5pt,color=red] (1,0) -- (1,-0.3);
		\draw[line width = 1.5pt,color=red]
		(1.3,0.3) -- (1.3,0.5);
		\draw[line width = 1.5pt,color=red] (1.3,0.3) -- (1.1,0.5);
		\draw[line width = 1.5pt,color=green] (1.3,0.3) -- (1.5,0.3);
		\draw[line width = 1.5pt] (1.5,0.3) -- (1.8,0.3);

		%Intersection

		\draw(0.55,0) node[below,color = red]{$f_1$};
		\draw(0.8,-0.1) node[below,color = red]{$f_2$};
		\draw(1.7,0) node[below,color = red]{$f_3$};
		\draw(1.5,0.3) node[below,color = red]{$f_4$};

		\draw(1.27,0) node[above,color=green]{$T$};

		\end{tikzpicture}
	\end{center}
	
\end{figure}

\begin{figure}[ht!]
	\begin{center}
		\begin{tikzpicture}[scale=6]
		
		\draw[line width = 1.5pt,color=green] (1,0) -- (1.3,0.3);
		\draw[line width = 1.5pt,color=green] (0.5,0) -- (1.7,0);
		\draw[line width = 1.5pt,color=green] (1.3,0.3) -- (1.5,0.3);
		\draw[line width = 1.5pt,color=green] (0.6,0) -- (0.72,-0.2);
		\draw(0.55,0) node[below,color = red]{$f_1$};
		\draw(0.62,-0.13) node[below,color = red]{$f_2$};
		\draw(1.7,0) node[below,color = red]{$f_3$};
		\draw(1.5,0.3) node[below,color = red]{$f_4$};
		\end{tikzpicture}
	\end{center}
	
\end{figure}

\begin{figure}[ht!]
	\begin{center}
		\begin{tikzpicture}[scale=6]
		
		\draw[line width = 1.5pt,color=blue] (1,0) -- (1.3,0.3);
		\draw[line width = 1.5pt,color=yellow] (0.5,0) -- (1.7,0);
		\draw[line width = 1.5pt,color=blue] (1.3,0.3) -- (1.5,0.3);
		\draw[line width = 1.5pt,color=violet] (0.6,0) -- (0.72,-0.2);
		\draw(0.55,0) node[below,color = red]{$f_1$};
		\draw(0.62,-0.13) node[below,color = red]{$f_2$};
		\draw(1.7,0) node[below,color = red]{$f_3$};
		\draw(1.5,0.3) node[below,color = red]{$f_4$};

		\draw(1.3,0) node[above,color = yellow]{$a_1$};
		\draw(0.8,-0.1) node[below,color = violet]{$a_2$};
		\draw(1.1,0.2) node[above,color = blue]{$a_3$};
		\end{tikzpicture}
	\end{center}
	
\end{figure}

\begin{figure}[ht!]
	\begin{center}
		\begin{tikzpicture}[scale=6]

		% Axes et point d
		
		\draw[line width = 1.5pt,color=blue] (1,0) -- (1.3,0.3);
		\draw[line width = 1.5pt,color=yellow] (0.5,0) -- (1.7,0);
		\draw[line width = 1.5pt,color=blue] (1.3,0.3) -- (1.5,0.3);
		\draw[line width = 1.5pt,color=violet] (0.6,0) -- (0.72,-0.2);
		
		\draw[dashed] (0,0) -- (0.25,0);
		
		\draw[line width = 1.5pt] (0.25,0) -- (0.5,0);
		\draw[line width = 1.5pt] (0,0.4) -- (0.35,0);
		
		\draw[line width = 1.5pt] (1.7,0) -- (2,0);
		\draw[dashed] (2,0) -- (2.2,0);

		\draw[line width = 1.5pt,color=red] (0.7,0) -- (0.7,0.3);
		\draw[line width = 1.5pt,color=red] (0.8,0) -- (0.8,0.2);
		\draw[line width = 1.5pt,color=red] (0.85,0) -- (0.85,0.08);
		\draw[line width = 1pt,color=red] (0.87,0) -- (0.87,0.04);
		\draw[line width = 0.8pt,color=red] (0.89,0) -- (0.89,0.02);

		\draw[line width = 1.5pt] (0.72,-0.2) -- (0.74,-0.24);
		\draw[line width = 1.5pt] (0.74,-0.24) -- (0.68,-0.24);
		\draw[line width = 1.5pt] (0.74,-0.24) -- (0.8,-0.24);
		\draw[dashed] (0.68,-0.24) -- (0.3,-0.24);

		\draw[line width = 1.5pt,color=red] (1,0) -- (1.3,-0.3);
		\draw[line width = 1.5pt,color=red] (1,0) -- (1,-0.3);
		\draw[line width = 1.5pt,color=red] (1.3,0.3) -- (1.3,0.5);
		\draw[line width = 1.5pt,color=red] (1.3,0.3) -- (1.1,0.5);
		
		\draw[line width = 1.5pt] (1.5,0.3) -- (1.8,0.3);

		\draw(0.55,0) node[below,color = red]{$f_1$};
		\draw(0.62,-0.13) node[below,color = red]{$f_2$};
		\draw(1.7,0) node[below,color = red]{$f_3$};
		\draw(1.5,0.3) node[below,color = red]{$f_4$};

		\draw(1.3,0) node[above,color = yellow]{$a_1$};
		\draw(0.8,-0.1) node[below,color = violet]{$a_2$};
		\draw(1.1,0.2) node[above,color = blue]{$a_3$};		
		%Projets
		\end{tikzpicture}
	\end{center}
	
\end{figure}

\begin{figure}[ht!]
	\begin{center}
		\begin{tikzpicture}[scale=6]

		% Axes et point d
		
		\draw[line width = 1.5pt,color=blue] (1,0) -- (1.3,0.3);
		\draw[line width = 1.5pt,color=yellow] (0.5,0) -- (1.7,0);
		\draw[line width = 1.5pt,color=blue] (1.3,0.3) -- (1.5,0.3);
		\draw[line width = 1.5pt,color=violet] (0.6,0) -- (0.72,-0.2);
		
		\draw[dashed] (0,0) -- (0.25,0);
		
		\draw[line width = 1.5pt] (0.25,0) -- (0.5,0);
		\draw[line width = 1.5pt] (0,0.4) -- (0.35,0);
		
		\draw[line width = 1.5pt] (1.7,0) -- (2,0);
		\draw[dashed] (2,0) -- (2.2,0);

		\draw[line width = 1.5pt,color=yellow] (0.7,0) -- (0.7,0.3);
		\draw[line width = 1.5pt,color=yellow] (0.8,0) -- (0.8,0.2);
		\draw[line width = 1.5pt,color=yellow] (0.85,0) -- (0.85,0.08);
		\draw[line width = 1pt,color=yellow] (0.87,0) -- (0.87,0.04);
		\draw[line width = 0.8pt,color=yellow] (0.89,0) -- (0.89,0.02);

		\draw[line width = 1.5pt] (0.72,-0.2) -- (0.74,-0.24);
		\draw[line width = 1.5pt] (0.74,-0.24) -- (0.68,-0.24);
		\draw[line width = 1.5pt] (0.74,-0.24) -- (0.8,-0.24);
		\draw[dashed] (0.68,-0.24) -- (0.3,-0.24);

		\draw[line width = 1.5pt,color=yellow] (1,0) -- (1.3,-0.3);
		\draw[line width = 1.5pt,color=blue] (1,0) -- (1,-0.3);
		\draw[line width = 1.5pt,color=blue] (1.3,0.3) -- (1.3,0.5);
		\draw[line width = 1.5pt,color=blue] (1.3,0.3) -- (1.1,0.5);
		
		\draw[line width = 1.5pt] (1.5,0.3) -- (1.8,0.3);

		\draw(0.55,0) node[above,color = red]{$f_1$};
		\draw(0.6,-0.13) node[below,color = red]{$f_2$};
		\draw(1.7,0) node[below,color = red]{$f_3$};
		\draw(1.5,0.3) node[below,color = red]{$f_4$};

		\draw(1.3,0) node[above,color = yellow]{$b_1$};
		\draw(0.8,-0.1) node[below,color = violet]{$b_2$};
		\draw(1.1,0.2) node[above,color = blue]{$b_3$};

		%Projets

		\end{tikzpicture}
	\end{center}
	
\end{figure}
\medskip
\medskip

\noindent By applying this same decomposition on every $V_i$, we get a collection $(b_i)_{i=1}^s$ of connected open subsets of $X$ such that

\begin{enumerate}
	\item $ X = \bigcup_{i=1}^s \overline{b_i},$ 
	\item $\mathrm{diam}(b_i) < \delta$ since each $b_i$ is a subset of one $V_j$,
	\item $card (\overline{b_j} \cap \overline{b_i}) \leq 1$,
	\item $card(\overline{b_j}\backslash b_j) \leq 2.$
\end{enumerate}
%\end{proof}

\noindent We are now going to make the proof with step functions, and conclude by a density argument. We consider 
$$ \psi_{(b_i)}  : \begin{cases} 
\: [0,1] \to  \mathbb{R}\\
\: x \longmapsto 1 \text{ if } x \in b_i,\\
\: x \longmapsto \frac{1}{ord(x)} \text{ if } x \in \partial b_i ,\\
\: x \longmapsto 0 \text{ otherwise}.

\end{cases} .$$

\noindent  For each $b_i$, we take $y_i \in b_i$ and set $c_i = \varphi(y_i)$ . We set $\varphi_0 = \sum_{i=0}^{r} c_i \psi_{(b_i)}$ with $c_i, i=0,\dots, r$. Since $diam(b_i)< \delta$, we have for $x \in X$
$$ \big| \varphi(x)-\varphi_0(x)\big| \leq \varepsilon .$$

\noindent Since $w_f(x)$ is infinite, for each $k$, there is one dendrite $f^i(D^k)$ such that $$w_f(x) \cap \textrm{int} \big(f^i(D_k)\big) \neq \emptyset .$$
Then there exists $n_0 \in \mathbb{N}$ such that $$f^{n_0}(x) \in \textrm{int} \big(f^i(D_k)\big).$$ But with the properties of $(f^j(D_k))_{j=1,\dots,\alpha_k}$, we have for $m \in \mathbb{N}$
$$ f^{n_0+m}(x) \in f^{i+m\textrm{ mod } \alpha_k}(D_k).$$
\noindent We can thus write

\begin{align*}
S_N(x, \psi_{(b_i)}) &= \frac{1}{N} \sum_{n<n_0} \mu(n) \psi_{(b_i)}(f^n(x))+\frac{1}{N}\sum_{n=n_0}^N \mu(n) \psi_{(b_i)}(f^n(x))\\
&= o(1) + \sum_{j=0}^{\alpha_k-1}\frac{1}{N}\sum_{n_0}^N \mu(n) \psi_{(b_i)}(f^n(x))\mathds{1}_{f^j(D_k)}(f^n(x)).
\end{align*}

\noindent Then we show that for $j\in \{0,\dots,\alpha_k-1\}, x_n^j:=\psi_{(b_i)}(f^n(x))\mathds{1}_{f^j(D_k)}(f^n(x))$ for $n\geq n_0$ almost behaves like a periodic sequence and control the other terms. Indeed, if $f^j(D_k)$ does not contain any of the points of $\partial b_i$, then $(x_n^j)_{n\geq n_0}$ is either constantly equal to 0 ( if $f^j(D_k)\cap b_i=\emptyset$) or periodic ( if $f^j(D_k)\subset b_i
$ ). Then, from Proposition \ref{periodic}, in both cases, we have 
$$ A_N^j := \frac{1}{N}\sum_{n=n_0}^N \mu(n) \psi_{(b_i)}(f^n(x))= \frac{1}{N} \sum_{n=n_0}^N \mu(n)x_n^j = o(1).$$ 

\noindent If now on the contrary, $f^j(D_k)$ contains one or two points of $\partial b_i$, we can notice that if $x_n^j \neq 0$ and $x_m^j \neq 0$ then $n=m$ or $|n-m|\geq \alpha_k$. By Lemma \ref{holed} once again, we have

$$\underset{N \to \infty}{\limsup}A_N^j \leq \frac{1}{\alpha_k}.$$
There at most two intervals $f^j(D_k),f^l(D_k)$ in which this situation can occur since $card(\partial b_i) \leq 2$. It implies that

$$ \underset{N \to \infty}{\limsup} |S_N(x, \psi_{(b_i)})|\leq
\underset{N \to \infty}{\limsup} \sum_{j=0}^{\alpha_k-1}|A_N^s|\leq \frac{2}{\alpha_k}.$$

\noindent Letting $k\to \infty$, we obtain the desired result.

\section{\bf Remarks on Li-Oprocha-Zhang's reduction theorem.}\label{SV}
%The question of Li-Oprocha-Zhang was motivated by the recent attractive research on Sarnak M\"{o}bius disjointedness conjecture. Roughly speaking,\\   

%It follows from our main result combined with the result from \cite{HMY} that the Sarnak M\"{o}bius disjointedness conjecture holds for the class of dendrite with countable endpoints and zero entropy.\\ 
In this section, we recall the reduction of Sarnak M\"{o}bius disjointness to the Gehman dendrites obtained by J. Li, P. Oprocha and G-H. Zhang. Precisely, they established the following
\begin{thm}[\cite{JOZ}, Corollary 1.7]
	The Sarnak's conjecture is true on every dynamical system with zero topological entropy if and only if it is so for all surjective dynamical systems over the Gehman dendrite with zero topological entropy.
\end{thm}

\noindent We recall that the Gehman dendrite is the topologically unique dendrite whose set of endpoints is homeomorphic to the Cantor set and whose branch points are all of order $3$.
\medskip
\noindent The principal ingredient in their proof is the following
%\begin{center}
%	\includegraphics[scale=0.7]{gehman_dendrite.png}
%	\captionof{figure}{The Gehman dendrite}
%\end{center}
\begin{thm}
	Every one sided-subshift can be embedded to a surjective dynamical system on the Gehman dendrite with the same topological entropy.
\end{thm}

%In \cite{boyle}, it is shown that every dynamical system with zero topological entropy is a factor of a two sided subshift with zero topological entropy.  We then get back to a one-sided subshift. 

\noindent According to the result by M. Boyle, D. Fiebig and U.Fiebig \cite{boyle}, the principal idea is to use the fact that every dynamical system of zero entropy can be seen as a factor of a two sided subshift with zero topological entropy, then get back to a one-side subshift with zero topological entropy. For more details, we refer to \cite{JOZ},

%We just give the idea. Let $D_*$ be the Gehman dendrite. Since the endpoints is a Cantor set, we construct a continuous function $f: D_* \to D_*$ which acts as the shift on the endpoints and such that $f^m($
%Then, if we take $(v, \sigma)$ a one-sided subshift , it can be seen as subsystem of $(\Sigma, f)$ ( and thus of ($D,f$) for $f=f_*^m)$ for some $m \in \mathbb{N}$ (take $m \geq 2^n$ with $n$ the cardinal of the alphabet of $V$).
%Restricting $D_*$ to a set appropriate to $V$, it can be easily shown that the topological entropy is not modified since what %happens outside of $V$ all converges to the root and does not produce any chaos.
\medskip

% Therefore, besides of ergodic invariant measures on $(V, \sigma)$ , any other ergodic invariant measure on $D_*$ has the fixed point $c$ as its support.

%Therefore, this result proves that the Gehman dendrite can contain any dynamical system and thus it is enough to prove Sarnak's conjecture on dendrites. However, all of the proofs on dendrites presented in this thesis take the hypothesis of a set of endpoints which is countable whereas the fact the set of endpoints of the Gehman dendrite is uncountable is fundamental in the proof above. This implies that despite the great results presented, it is necessary to re-adapt completely the proof and come up with new lemmas when the set of endpoints is uncoutable.

It follows that in the class of Gehman dendrite with zero topological entropy the spectrum can be any spectrum realized by any map with zero topological entropy. Nervelessness, we ask %the following weak form of  J. Li, P.Oprocha and G-H. Zhang question
 
\begin{Que} Let $(X,f)$ be a dynamical system with zero topological entropy and endpoints countable, do we have that its spectrum is singular. 
\end{Que}

\textbf{Acknowledgments.} This paper is a part of the Master thesis by the second author. The authors would like to thanks G. Askri, K.  Dajani, H. Marzougui, M. Nerurkar, O. Sarig, and X-D. Ye for their comments and suggestions.

\bibliographystyle{amsplain}

\end{document}